\PassOptionsToPackage{log-declarations=false}{xparse}
\RequirePackage{luatex85}
\documentclass[11pt,a4paper,final,twoside]{amsart}
\usepackage{ifthen}
\usepackage[T1]{fontenc}
\usepackage{amssymb}
\usepackage{bbm}
\usepackage{ellipsis}
\usepackage{microtype}
\usepackage{tikz}
\usetikzlibrary{decorations.markings}
\usetikzlibrary{arrows}
\expandafter\let\csname ver@amsthm.sty\endcsname\relax
\let\theoremstyle\relax

\let\qedhere\relax

\usepackage{hyperxmp}
\usepackage[unicode,pdftex,pdfusetitle]{hyperref}
\hypersetup{draft=false,colorlinks=true,pdfpagemode=UseOutlines,%
            plainpages=false,
            breaklinks=true,bookmarksnumbered=true,pdfborder=0 0 0,%
            pdfkeywords={motivic,cellular,spherical},%
            pdfauthor={Konrad Voelkel},%
            pdfsubject={MSC2010: 14F42, 20Gxx.
              License: CC BY-NC-ND 4.0 Germany.
            }}
\usepackage{amsthm}
\usepackage[nameinlink]{cleveref}
\theoremstyle{plain}
\newtheorem{theorem}{Theorem}
\newtheorem*{theorem*}{Theorem}
\newtheorem{proposition}[theorem]{Proposition}
\newtheorem{lemma}[theorem]{Lemma}
\newtheorem{corollary}[theorem]{Corollary}
\newtheorem*{lemma*}{Lemma}
\newtheorem*{corollary*}{Corollary}

\newtheorem*{conjecture*}{Conjecture}
\theoremstyle{definition}
\newtheorem{defn}[theorem]{Definition}
\crefname{defn}{definition}{definitions}
\Crefname{defn}{Definition}{Definitions}
\crefname{thmqs}{}{XXX}
\creflabelformat{thmqs}{#2{Quillen--Suslin}#3}
\crefname{thmhp}{}{XXX}
\creflabelformat{thmhp}{#2{homotopy purity}#3}
\newtheorem{example}[theorem]{Example}
\newtheorem*{defn*}{Definition}

\newtheorem*{convention*}{Convention}
\theoremstyle{remark}
\newtheorem{remark}[theorem]{Remark}
\newtheorem*{warning}{Warning}
\newenvironment*{explain}{\quad\\}{}
\newcommand{\set}{\mathbb}
\newcommand{\Natural}{\set{N}}
\newcommand{\Integer}{\set{Z}}

\newcommand{\epito}{\twoheadrightarrow}
\newcommand{\monoto}{\hookrightarrow}
\newcommand{\monofrom}{\hookleftarrow}
\newcommand{\isoto}{{\ \longrightarrow\hspace{-15pt}{}^\sim\hspace{8pt}}}

\newcommand{\closedto}{{\,\monoto\hspace{-0.45cm}{\shortmid}\hspace{0.25cm}}}
\newcommand{\closedfrom}{{\,\monofrom\hspace{-0.35cm}{\shortmid}\hspace{0.25cm}}}
\newcommand{\opento}{{\,\monoto\hspace{-0.45cm}{{}^{{}_\circ}}\hspace{0.25cm}}}

\newcommand{\setbar}{{\ |\ }}
\renewcommand{\tilde}{\widetilde}
\renewcommand{\phi}{\varphi}

\DeclareMathOperator{\hocolim}{hocolim}
\DeclareMathOperator{\hocofib}{hocofib}
\DeclareMathOperator{\Spec}{Spec}
\DeclareMathOperator{\HP}{HP}

\DeclareMathOperator{\OP}{OP}
\DeclareMathOperator{\Gr}{Gr}

\DeclareMathOperator{\Th}{Th}
\DeclareMathOperator{\SpGr}{SpGr}

\DeclareMathOperator{\Ho}{\mathcal{H}o}

\DeclareMathOperator{\sSet}{sSet}
\DeclareMathOperator{\Spc}{Spc}

\DeclareMathOperator{\GL}{GL}
\DeclareMathOperator{\Sp}{Sp}

\DeclareMathOperator{\Ffour}{F_{4}}

\DeclareMathOperator{\Esix}{E_{6}}
\DeclareMathOperator{\Spin}{Spin}
\DeclareMathOperator{\Sing}{Sing}
\DeclareMathOperator{\SO}{SO}

\DeclareMathOperator{\A}{\mathbb{A}}
\DeclareMathOperator{\Ao}{\mathbb{A}^1}
\DeclareMathOperator{\Po}{\mathbb{P}^1}
\DeclareMathOperator{\Gm}{\mathbb{G}_m}

\numberwithin{subsection}{section}
\author{Konrad Voelkel}
\title{Motivic Cell Structures for Spherical Varieties}
\begin{document}

\begin{abstract}
In this note, we give a general method to obtain unstable motivic cell structures, following Wendt's application~\cite{wendtcell10} of the Bia{\l}ynicki-Birula algebraic Morse theory. 
We then apply the method to spherical varieties, with special attention to the case of rank $1$, to obtain unstable motivic cell structures after a finite number of $\Po$-suspensions.
This refines the toolkits of Dugger--Isaksen and Wendt.
\end{abstract}

\maketitle

We give an explicit proof (on \cpageref{proofof:spherical-stably-cellular}) of
\begin{theorem}\label{thm:spherical-stably-cellular}
  Let $k$ be a field and $X$ a spherical $k$-variety,
  then $X$ is motivic stably cellular in the sense of Dugger--Isaksen. 
\end{theorem}
\begin{corollary*}\label{cor:spherical-mtm}
The Voevodsky motive of a spherical variety is mixed Tate.
\end{corollary*}

Dugger and Isaksen chose to work mostly with stable cellularity because products of unstably cellular spaces are not necessarily unstably cellular (see~\cref{remark:products-bad-for-cells}).
We refine this investigation by taking only a finite amount of suspensions.
\begin{defn*}
Let $X$ be a motivic space. If the simplicial suspension $\Sigma^k X$ is motivic unstably cellular, we say $X$ is \emph{$k$-suspended cellular}.
\end{defn*}

To gain more control on cell structures, we prove (on \cpageref{proofof:cells-from-completion})
\begin{theorem}\label{thm:cells-from-completion}
  Let $X$ be a smooth $S$-variety and let $X \opento \overline{X}$ be a closed immersion into a smooth $S$-variety $\overline{X}$ with complement $D = \bigcup_{i=1}^n D_i$ a divisor with $n$ irreducible smooth components $D_i$. 
  If $\overline{X}$ is $n$-suspended cellular and each $D_i$ is atacc (see~\cref{defn:totally-affinely-contractible}), then $X$ is $(n+k)$-suspended cellular.
\end{theorem}
\begin{corollary*}[{\Cref{thm:unstable-cells-after-single-suspension}}]
  Let $X$ be a homogeneous space under a split reductive group $G$ with a $G$-equivariant completion $X \opento \overline{X}$ such that the complement $D$ is irreducible (a two-orbit completion).
  Then $X$ is $1$-suspended cellular.
\end{corollary*}
\begin{conjecture*}
 A two-orbit completable homogeneous space is unstably cellular.
\end{conjecture*}
This is known for affine split quadrics $AQ_{2n} = \SO_{n,n+1}/\SO_n\times \SO_{n+1}$ by~\cite[Theorem 2.2.5]{asokdoranfasel15}, $AQ_{2n-1} = \SO_{n,n}/\SO_n\times\SO_n$ by a well-known elementary argument (\cref{lemma:odd-quadrics}) and quaternionic projective space $\HP^n = \Sp_{2n+2}/\Sp_2\times\Sp_{2n}$ by~\cite[Theorem 4.4.8]{voelkel16} (with the idea coming from Panin--Walter~\cite{paninwalter10a}). In future work, we intend to discuss the other two-orbit completions, like the Cayley plane $\OP^2$, in more detail.

On \cpageref{proofof:subspace-arrangements}, we prove
\begin{theorem}\label{thm:subspace-arrangements}
  Let $\mathcal{A} = \{L_i\}_{i=0}^r$ an arrangement of $r+1$ linear subspaces $L_i \closedto \mathbb{A}^n$ of a fixed affine space $\mathbb{A}^n$.
  Then the complement $X := \mathbb{A}^n \setminus \bigcup \mathcal{A}$ is $r$-suspended cellular.
\end{theorem}
\begin{corollary*}
  A rank $r$ split torus $T = \mathbb{G}_m^{\times r}$ is $(r-1)$-suspended cellular.
\end{corollary*}

\begin{convention*}
We use the notation $Z \closedto X$ for a closed immersion of $Z$ into $X$ and $U \opento X$ for an open immersion of $U$ into $X$ throughout.
\end{convention*}
\begin{explain}
Two results not due to the author were not easily available in the literature,
which is why we provide a proof here: the folklore \cref{lemma:bundles-are-equivalences} that affine Zariski bundles are $\Ao$-equivalences and Wendt's \cref{lemma:vector-bundles-are-sharp} that vector bundle projections are sharp maps.
\end{explain}

\begin{explain}
\textbf{Acknowledgments.}
I want to thank Matthias Wendt for his guidance and the FRIAS research focus ``Cohomology in Algebraic Geometry and Representation Theory'' for hospitality.
Except for \cref{thm:subspace-arrangements}, this work is a partial derivative of the first two chapters of the author's 2016 PhD thesis under supervision of Matthias Wendt at the University of Freiburg, Germany.
\end{explain}


\section{Motivic Spaces}\label{section:mothot}

\begin{explain}
  We model motivic spaces $\Spc(S)$ with simplicial presheaves on smooth finite type schemes over a Noetherian base scheme $S$, with the $\Ao$-local Nisnevich-local injective model structure. We always consider \emph{pointed} motivic spaces.
  This model category is well explained by Dugger~\cite{dugger01} and Dugger--Hollander--Isaksen~\cite{duggerhollanderisaksen04}.
  The idea to approach a homotopy theory of algebraic varieties in this way was introduced by Morel and Voevodsky~\cite{voevodskymorel99}, building on work of Jardine.
\end{explain}

\begin{defn}
A space $X \in \Spc(S)$ such that the basepoint $S \to X$ is an $\Ao$-equivalence is called \emph{$\Ao$-contractible}.
\end{defn}

\begin{warning}
A space $X$ which is $\Ao$-contractible need not be an affine space itself.
An ample supply of quasi-affine non-affine varieties which are $\Ao$-contractible is given by Asok and Doran~\cite{asokdoran07},~\cite{asokdoran08}.
Duboulouz and Fasel gave examples of smooth affine threefolds over fields of characteristic $0$ which happen to be $\Ao$-contractible but are not isomorphic to affine spaces~\cite{duboulozfasel15}.
\end{warning}

\subsection{Bundles on Schemes}

One source of $\Ao$-equivalences are affine bundle projections:
\begin{defn}\label{defn:affine-bundle}
  For $B$ a motivic space,
a Zariski locally trivial fiber bundle $p \colon E \epito B$ with fibers $p^{-1}(b)$ isomorphic to $\mathbb{A}^n$ is called an \emph{affine bundle} on $B$.
\end{defn}

\begin{lemma}\label{lemma:bundles-are-equivalences}
Let $p \colon E \epito B$ be an affine bundle of rank $n$, with $B$ smooth.
Then $p$ is an $\Ao$-weak equivalence.
\end{lemma}
\begin{proof}
By definition of a bundle, $B$ admits a Zariski cover $\mathcal{U} = {\{U_i \opento B\}}_{i \in I}$ and there are isomorphisms $\phi_i$ over $U_i$ from $U_i \times \mathbb{A}^n$ to $p^{-1} U_i$.
Choose (for convenience of stating the proof) a well-order on $I$.
For every word $\alpha = (i_1,\dots,i_k)$ of length $k$ over $I$ we let $\phi_\alpha$ be the restriction of $\phi_i$ to $U_\alpha := U_{i_1} \times_B \cdots \times_B U_{i_k}$ for $i = \min(\alpha)$.
As the \v{C}ech nerve of a cover is defined as $\check{C}^k(\mathcal{U}) = \coprod_{|\alpha|=k} U_\alpha$,
we get an isomorphism between $p^\ast \check{C}^k(\mathcal{U})$ and $\check{C}^k(\mathcal{U}) \times \mathbb{A}^n$ over $\check{C}^k(\mathcal{U})$, by disjoint union of the $\phi_\alpha$.

By definition of the $\Ao$-weak equivalences, the projections $\check{C}^k(\mathcal{U}) \times \mathbb{A}^n \epito \check{C}^k(\mathcal{U})$ are $\Ao$-weak equivalences, and so is the morphism $p^\ast \check{C}^k(\mathcal{U}) \epito \check{C}^k(\mathcal{U})$.
As any degree-wise weak equivalence of simplicial objects is a weak equivalence,
$p^\ast \check{C}^\bullet(\mathcal{U}) \epito \check{C}^\bullet(\mathcal{U})$ is an $\Ao$-weak equivalence.
By definition, $p^\ast \check{C}^\bullet(\mathcal{U}) = \check{C}^\bullet(p^\ast \mathcal{U})$, where $p^\ast \mathcal{U} := {\{ p^\ast U_i \opento E\}}_{i \in I}$ is the induced Zariski cover of $E$.
As the homotopy colimit of a \v{C}ech nerve is the space covered~\cite[Theorem 1.2]{duggerhollanderisaksen04}, we get a commutative diagram in which we know that all morphisms except possibly $p$ are $\Ao$-weak equivalences:

{\centering
\begin{tikzpicture}[arrows=->]
\node (hpC) at (0,0) {$\hocolim(p^\ast \check{C}^\bullet(\mathcal{U}))$};
\node (hC) at (5,0) {$\hocolim(\check{C}^\bullet(\mathcal{U}))$};
\node (E) at (0,-1) {$E$};
\node (B) at (5,-1) {$B$};
\draw (E) to node[above] {$p$} (B);
\draw (hpC) to (hC);
\draw (hpC) to (E);
\draw (hC) to (B);
\end{tikzpicture}\\
}
From the diagram we see that $p$ is also an $\Ao$-weak equivalence.
\end{proof}

\begin{lemma}[Wendt]\label{lemma:vector-bundles-are-sharp}
Let $S$ be either a field or a Dedekind ring with perfect residue fields and $R$ a smooth $S$-algebra.
For $E$ and $B$ two $R$-varieties and $p \colon E \to B$ a rank $n$ vector bundle projection,
the underived pullback $p^\ast$ preserves homotopy colimits (for $\Ao$-local weak equivalences). 
For a diagram $D \in \Spc(R)/B$ we can compute
\[\hocolim p^\ast D \simeq p^\ast \hocolim D.\]
\end{lemma}
\begin{proof}
  Denote $\mathcal{E} \to B$ the associated frame bundle of $E \to B$, which is a $\GL_n$-principal bundle.
  Under the assumptions on $R$, the classifying space $B\Sing^{\Ao}_\bullet\GL_n$ is $\Ao$-local due to~\cite[Theorem 5.1.3 and the proof of Theorem 5.2.3]{ahw1}.
  As in~\cite[Proof of Theorem 4.6]{wendt11} for the case $G=\GL_n$,
  there is an $\Ao$-local fiber sequence $\GL_n \to \mathcal{E} \to B$,
  so in particular a simplicial fiber sequence.
  By Rezk's theorem~\cite[Theorem 4.1., (1) $\Leftrightarrow$ (3)]{rezk98} a map $\mathcal{E} \to B$ that induces a simplicial fiber sequence is sharp.
  Since the $\Ao$-local injective model category is right proper,
  pullback along a sharp map preserves homotopy colimits~\cite[Proposition 2.7]{rezk98}.
  Zariski-locally, we can recover the (homotopy) pullback along $E\to B$ from the pullback along $\mathcal{E} \to B$,
  hence globally the homotopy pullback is given by the underived pullback along $E \to B$. To see the last claim, compute the homotopy pullbacks of $X \to B$ by fibrant replacement of $X \to B$.
\end{proof}

\subsection{Motivic Spheres}

\begin{defn}
For $p,q \in \Integer$ with $p \geq q$, the motivic space
\[\S^{p,q} := {\left( {\Gm}_{,S} \right)}^{\wedge q} \wedge \S^{p-q} \in \Spc(S)\]
is called a \emph{motivic sphere}.
Here $\S^{p-q} := {\left(\S^1\right)}^{\wedge p-q} \in \sSet$ is a simplicial sphere, where
$\S^1 := \Delta^1/\partial \Delta^1$ (pointed by $\partial \Delta^1$) and ${\Gm}_{,S} := \Gm \times_\Integer S$ is the multiplicative group scheme over $S$ (pointed by the unit), where $\Gm(R)=R^\times$ for any ring $R$.
\end{defn}

\begin{example}
There is an $\Ao$-homotopy equivalence
\[ (\Po,0) \isoto \S^{2,1},\]
that is, an isomorphism in the homotopy category of $\Spc(S)$~\cite[Example 3.2.18]{voevodskymorel99}.
One can see this directly by writing $\Po = X \cup Y$ with $X = \Po \setminus \{0\}$ and $Y = \Po \setminus \{\infty\} \isoto X$,
so that $X \times_{\Po} Y = \Ao \setminus \{0\} \simeq \S^{1,1}$
and $X$ and $Y$ are both $\Ao$-contractible.
\end{example}

\begin{example}[{\cite[Example 3.2.20]{voevodskymorel99}}]\label{example:affine-space-origin-removed}
Affine space without origin is a motivic sphere:
\[\mathbb{A}^n \setminus \{0\} \simeq \S^{2n-1,n}\]
\end{example}

For odd-dimensional split quadrics, there is a well-known elementary argument to see that they are motivic spheres:
\begin{lemma}\label{lemma:odd-quadrics}
Let $AQ_{2n-1} := \{ (x,y) \in \A^n \times \A^n \setbar \sum_{i=1}^n x_i y_i = 1 \}$ (considered as affine algebraic variety over $\Integer$),
then $\pi \colon AQ_{2n-1} \epito \A^n \setminus \{0\}$ given by $(x,y) \mapsto y$ is a rank $n-1$ affine bundle and over any base scheme $S$ there is an isomorphism
\[AQ_{2n-1} \isoto \S^{2n-1,n}.\]
\end{lemma}
\begin{proof}
We can cover $\A^n \setminus \{0\}$ by the varieties $U_i := \{y_i \neq 0\}$,
over which $\pi^{-1}(U_i) = \{ (x,y) \in \A^n \times \A^n \setbar y_i \neq 0,\ \sum_{j=1}^n x_j y_j = 1 \}$
can be rewritten as
\[\pi^{-1}(U_i) = \left\{ (x,y) \in \A^n \times \A^n \ \Big{\vert}\  y_i \neq 0,\ x_i = y_i^{-1}\left({1 - \sum_{j=1,\ j\neq i}^n x_j y_j} \right) \right\}\]
so there are isomorphisms
\[\pi^{-1}(U_i) \isoto \A^{n-1} \times U_i,\qquad (x,y) \mapsto ((x_1,\dots,\hat{x_i},\dots,x_n),y).\]
For fixed $y \in \A^n \setminus \{0\}$, the equation for $x_i$ is linear in $x$ with $x_i$ removed, which guarantees that the inverse map is linear in the $\A^{n-1}$-component.

Now apply \cref{lemma:bundles-are-equivalences} and \cref{example:affine-space-origin-removed}.
\end{proof}


\section{Motivic Cell Structures}\label{section:motcells}

\begin{defn}[{Dugger and Isaksen~\cite[Definition 2.1]{duggerisaksen05}}]\label{defn:motcells}
Let $\mathcal{M}$ be a pointed model category and $\mathcal{A} \subset \mathop{Ob}\mathcal{M}$ a set of objects.
The class of \emph{$\mathcal{A}$-cellular objects} in $\mathcal{M}$ is defined as the smallest class of objects containing $\mathcal{A}$ that is closed under weak equivalence and contains all homotopy colimits over diagrams whose objects are all $\mathcal{A}$-cellular.
\end{defn}

\begin{defn}\label{defn:motcellsconcrete}
For the pointed model category of pointed motivic spaces $\Spc(S)$ let $\mathcal{A} := \{\S^{p,q} \setbar p,q \in \Natural,\ p \geq q\}$ be the set of motivic spheres.
The $\mathcal{A}$-cellular objects in $\Spc(S)$ are called \emph{motivically cellular}.
A motivic space $X$ with $k$-fold simplicial suspension $\Sigma^k X$ motivically cellular is called \emph{$k$-suspended cellular}.
For motivic spectra with $\mathcal{A}^s := \{\S^{p,q} \setbar p,q \in \Integer\}$,
we call $\mathcal{A}^s$-cellular objects \emph{stably motivically cellular}.
A motivic space $X \in \Spc(S)$ with $\Sigma^\infty_+ X$ stably motivically cellular is also called stably motivically cellular.
\end{defn}

\begin{defn}
  Given a morphism $f \colon X \to Y$ of motivic spaces, we define the \emph{homotopy cofiber} $\hocofib(f)$ as the homotopy colimit of the solid-lines diagram:
  
{\centering
\begin{tikzpicture}[arrows=->,scale=2]%
\node (X) at (-1,0) {$X$};
\node (point) at (-0.2,-0.4) {$\ast$};
\node (Y) at (0,0) {$Y$};
\node (hocofib) at (0.8,-0.4) {$hocofib(f)$};
\draw (X) to node[above] {$f$} (Y);
\draw (X) to (point);
\draw (point) [dotted] to (hocofib);
\draw (Y) [dotted] to (hocofib);
\end{tikzpicture}\\
}

We call a sequence $X \xrightarrow{f} Y \to Z$ in the homotopy category of $\Spc(S)$ a \emph{homotopy cofiber sequence} if the sequence is isomorphic to $X \xrightarrow{f} Y \to \hocofib(f)$ in the homotopy category.
\end{defn}

\begin{example}
Projective space $\mathbb{P}^n$ carries a motivic cell structure,
as there exists a homotopy cofiber sequence (compare~\cite[Proposition 2.13]{duggerisaksen05})
\[\mathbb{A}^n \setminus \{0\} \to \mathbb{P}^{n-1} \to \mathbb{P}^n\]
and $\mathbb{A}^n \setminus \{0\}$ is a motivic sphere $\S^{2n-1,n}$ (up to $\Ao$-homotopy equivalence~\cite[Example 2.11]{duggerisaksen05} for a proof of this claim first made by Morel and Voevodsky~\cite[Example 3.2.20]{voevodskymorel99}).
This homotopy cofiber sequence yields a distinguished triangle in the derived category of motives
\[M(\mathbb{P}^{n-1}) \to M(\mathbb{P}^n) \to \mathbbm{1}(n)[2n] \to\]
and one can show that the attaching map $\mathbb{A}^n \setminus \{0\} \to \mathbb{P}^{n-1}$ is $0$ at the level of motives for weight reasons,
hence there is a decomposition
\[M(\mathbb{P}^n) = \bigoplus_{i=0}^n \mathbbm{1}(i)[2i].\]
However, even on the level of spectra, in the stable motivic homotopy category,
the attaching map is non-trivial and $\bigvee_{i=0}^n \S^{2i,i}$ is a different motivic space with the same motivic decomposition as $\mathbb{P}^n$.
\end{example}

\begin{lemma}
If a motivic space $X$ admits a stable motivic cell structure,
its Voevodsky motive is of mixed Tate type.
\end{lemma}
\begin{proof}
  This follows directly from $M(S^{p,q}) = \mathbbm{1} \oplus \mathbbm{1}(q)[p]$ and the fact that the homotopy colimits defining the cell structure can be written as homotopy coequalizer and homotopy coproduct, which translates directly to distinguished triangles of mixed motives.
\end{proof}

\subsection{Motivic Thom Spaces}

\begin{theorem*}[Homotopy Purity, Morel and Voevodsky~{\cite[Thm. 3.2.23]{voevodskymorel99}}]\label[thmhp]{thm:hopurity}
For a closed immersion $\iota : Z \monoto X$ with open complement $U \monoto X$,
there is a natural homotopy cofiber sequence of pointed motivic spaces
\[ U \to X \to \Th(N_\iota)\]
where $\Th(N_\iota)$ denotes the \emph{Thom space} of the normal bundle $N_\iota$ of $\iota$,
which is defined using the zero section $Z \monoto N_\iota$
as $\Th(N_\iota) := N_\iota / {(N_\iota \setminus Z)} := \hocofib(N_\iota \setminus Z \monoto N_\iota)$.
\end{theorem*}

\begin{warning}
  It is not necessarily true that a Thom space over a motivically cellular base is again motivically cellular (it is not known whether counterexamples exist or whether we simply lack a proof).
  This is also unknown for stable motivic cell structures. 
\end{warning}

\begin{remark}\label{remark:trivial-thom-spaces}
For this reason, and also to be able to describe a cell structure explicitly, it is highly desirable to trivialize normal bundles.
Thom spaces over trivial bundles are just suspensions (\cite[Proposition 3.2.17]{voevodskymorel99}):
\[\Th(\mathbb{A}^n \times B \to B) = B_+ \wedge \S^{2n,n}.\]
\end{remark}

The following are tools to trivialize vector bundles.
\begin{theorem*}[Quillen--Suslin]\label[thmqs]{thm:quillen-suslin}
Let $R$ be a smooth finite type algebra over a Dedekind ring.
Then all algebraic vector bundles on $\mathbb{A}^n_R$ are extended from $\Spec(R)$.
\end{theorem*}
This is beautifully explained in Lam's book~\cite[Theorem III.1.8]{lam06}.

It has been used to obtain a generalization, which one may see as a corollary to the vector bundle classification of Asok--Hoyois--Wendt:
\begin{lemma}\label{lemma:no-bundles-on-contractibles}
There are no non-trivial vector bundles on a smooth affine finite type $\Ao$-contractible variety $X$ over a Dedekind ring with perfect residue fields or a field.
\end{lemma}
\begin{proof}
Let $f : \Spec(k) \isoto X$ be the isomorphism in the homotopy category of motivic spaces $\Ho\Spc(k)$ given by contractibility.
From~\cite[Theorem 5.2.3]{ahw1}, we know
\[{\left\{\text{rank }r\text{ vector bundles on } X\right\}}/_\simeq \isoto {[X,Gr_r]}_{\Ao} \underset{f^\ast}{\isoto} {[\ast,Gr_r]}_{\Ao} = 1.\qedhere\]
\end{proof}
This fails already for smooth non-affine quasi-affine varieties that are $\Ao$-contractible,
where one can give infinitely many counterexamples~\cite[Corollary 4.3.9]{asokdoranfasel15}.

\begin{corollary}\label{cor:hopurity-first-application}
If $V \monoto M$ is a codimension $c$ closed immersion of smooth varieties over a smooth finite type $\Integer$-algebra $R$,
and the complement $M \setminus V$ is an $\Ao$-contractible smooth affine $R$-variety,
and $V$ is the total space of a vector bundle $V \epito M'$ with $M'$ an $\Ao$-contractible $R$-variety,
$M$ is a motivic sphere $\S^{2c,c}$.
\end{corollary}
\begin{proof}
Assume that $M'$ is smooth affine as well.
Using \cref{lemma:no-bundles-on-contractibles},
the vector bundle $V \epito M'$ is trivial, so the total space $V$ is also smooth affine $\Ao$-contractible.
Using \cref{lemma:no-bundles-on-contractibles} again,
the normal bundle $N_\iota \epito V$ is trivial.
The Thom space of a trivial bundle of rank $r$ over a base $B$ is $\Ao$-homotopy equivalent to ${(\Po)}^{\wedge r} \wedge B_+$.
We conclude by using \cref{thm:hopurity}, which hands us a homotopy cofiber sequence
\[M\setminus V \to M \to \Th(N_\iota).\]
Contractibility of $M\setminus V$ implies that $M \to \Th(N_\iota)$ is a weak equivalence,
so $M \simeq {(\Po)}^{\wedge c} \wedge \S^0 \simeq \S^{2c,c}$.

Now if $M'$ is not smooth affine, $N_\iota \epito V$ may be non-trivial.
Since $V \epito M'$ is still a weak equivalence, $V$ is still $\Ao$-contractible,
hence $\Th(N_\iota) \simeq {(\Po)}^{\wedge c} \wedge \S^0$, as Thom spaces are invariant under $\Ao$-equivalence by definition.
\end{proof}

\begin{defn}\label{defn:totally-affinely-contractible}
For a variety $N$, a Zariski cover $\mathcal{U} = {(U_i \opento N)}_{i \in I}$ (with $N = \bigcup_{i \in I} U_i$) is called \emph{totally cellular} if the \v{C}ech nerve $\check{C}^\bullet(\mathcal{U})$ is a simplicial object in cellular varieties.
It is called \emph{totally contractible} if the $U_\alpha = \bigcap_{j \in J} U_j$ for each $J \subset I$ are $\Ao$-contractible.
It is called \emph{totally affinely contractible} if there are affine bundles $\tilde{U}_\alpha \to U_\alpha$ with affine total spaces $\tilde{U}_\alpha \cong \mathbb{A}^{m_\alpha}$, compatible with the simplicial structure on $\check{C}^\bullet(\mathcal{U})$ (assembling to an affine bundle $\check{C}^\bullet(\tilde{\mathcal{U}}) \to \check{C}^\bullet(\mathcal{U})$).
A variety that admits a totally affinely contractible (Zariski) cover is called \emph{atacc} for short.
\end{defn}
The definition of total cellularity was made in the stable context by Dugger and Isaksen~\cite[Definition 3.7]{duggerisaksen05},
see also~\cite[Lemma 3.8]{duggerisaksen05}.

From the definitions follows immediately
\begin{proposition}\label{prop:totally-cellular-is-unstably-cellular}
  A totally contractible Zariski cover is totally cellular and a totally affinely contractible Zariski cover is totally contractible.
  A variety admitting a totally cellular Zariski cover (in particular, an atacc variety) is unstably cellular.
\end{proposition}

\begin{remark}\label{remark:products-bad-for-cells}
While smash products of unstably cellular spaces are again unstably cellular,
Dugger and Isaksen already noticed~\cite[Example 3.5]{duggerisaksen05} that it is in general hard to show whether a cartesian product of cellular spaces is unstably cellular.
Since it is easy to show that cartesian products of stably cellular spaces are stably cellular,
they only prove that Thom spaces of bundles over a totally cellular base are stably cellular~\cite[Corollary 3.10]{duggerisaksen05}.
As we are interested in unstable cell structures on spaces which are iterated Thom spaces,
we need a stronger statement:~\cref{thm:unstable-thom-cells}.
\end{remark}

\begin{lemma}\label{lemma:compute-thom-along-equivalence}
Let $p \colon E \epito B$ be a vector bundle and $B' \to B$ an $\Ao$-weak equivalence.
Then there exists 
a weak equivalence of Thom spaces
\[\Th(p) \isoto \Th(p').\]
\end{lemma}
\begin{proof}
Since vector bundle projections are sharp (\cref{lemma:vector-bundles-are-sharp}), the morphism $E \times_B B' \to E$ is a weak equivalence.
Let $s$ be the zero section of $p$ and $s'$ the zero section of the base change $p' \colon E \times_B B' \to B'$.
By construction of $s'$, we get a weak equivalence of $E \times_B B' \setminus s'(B')$ with $E \setminus s(B)$.
We proved that the diagrams whose homotopy colimits are $\Th(p)$ respectively $\Th(p')$ are weakly equivalent.
\end{proof}

\begin{corollary}
Let $p \colon E \epito B$ be a rank $n$ vector bundle and $B' \to B$ an affine bundle with $B' \cong \mathbb{A}^m$ (as varieties).
Then $\Th(p) \isoto B_+ \wedge \S^{2n,n} \simeq \S^{2n,n}$.
\end{corollary}
\begin{proof}
We use \cref{lemma:bundles-are-equivalences} to apply \cref{lemma:compute-thom-along-equivalence} and then \cref{remark:trivial-thom-spaces}. 
\end{proof}

\begin{theorem}\label{thm:unstable-thom-cells}
Let $p \colon E \to B$ be an algebraic vector bundle of rank $r$ and $\mathcal{U} = {(U_i \opento B)}_{i \in I}$ a totally affinely contractible Zariski cover of $B$, all defined over a ring $R$ which is smooth and finite type over a Dedekind ring with perfect residue fields or a field.
Then $B$, $E$ and $\Th(p)$ are unstably cellular.
\end{theorem}
\begin{proof}
  Unstable cellularity of $B$ is~\cref{prop:totally-cellular-is-unstably-cellular},
  cellularity of $E$ follows from~\cref{lemma:bundles-are-equivalences}, so it remains to show cellularity of the Thom space $\Th(p)$.
We use the morphism $l \colon \check{C}^\bullet(\mathcal{U}) \to B$ which induces a weak equivalence on homotopy colimits, i.e.\ 
$\hocolim\left(\check{C}^\bullet(\mathcal{U}) \right) \simeq B$.
The bundle $q \colon E \setminus B \to B$ obtained as sub-bundle of $p$ is a fiber bundle with fiber $\mathbb{A}^{r} \setminus \{0\}$.
The following diagram commutes:\\
{\centering
\begin{tikzpicture}[arrows=->]
\node (EmB) at (0,0) {$E \setminus B$};
\node (E) at (3,0) {$E$};
\node (Th) at (6,0) {$\Th(p)$};
\draw (EmB) to node[above] {$i$} (E);
\draw (E) to (Th);
\node (qU) at (0,2) {$q^\ast \check{C}^\bullet(\mathcal{U})$};
\node (pU) at (3,2) {$p^\ast \check{C}^\bullet(\mathcal{U})$};
\node (hcC) at (6,2) {$\hocofib(l^\ast i)$};
\draw (qU) to node[above] {$l^\ast i$} (pU);
\draw (qU) to node[left] {$q^\ast l$} (EmB);
\draw (pU) to node[left] {$p^\ast l$} (E);
\draw (pU) to (hcC);
\draw (hcC) [dashed] to (Th);
\end{tikzpicture}\\
}
The rows are homotopy cofiber sequences.
The middle column is a weak equivalence by \cref{lemma:vector-bundles-are-sharp}.
The left column is also a weak equivalence,
as it is the restriction of the middle column and the model structure is proper. 
(alternatively one could argue that spherical bundle projections are as sharp as vector bundle projections).
We inspect the first row more closely.
While the bundle $E$ might not trivialize over $U_i$,
its pullback to an affine space $\tilde{U}_i$
(given by the property of $U_i$ being totally affinely contractible) is trivial,
by \cref{thm:quillen-suslin}.
The same holds for each $U_\alpha$ with obvious definition of $\tilde{U}_\alpha$.
By \cref{lemma:compute-thom-along-equivalence} the Thom space of $E|_{U_i}$ is weakly equivalent to the Thom space of the pulled back bundles $E_{|\tilde{U}_i}$.
Now we can form a diagram, commutative up to homotopy\\
{\centering
\begin{tikzpicture}[arrows=->]
\node (qUa) at (0,2) {$q^\ast U_\alpha$};
\node (pUa) at (3,2) {$p^\ast U_\alpha$};
\node (hcCa) at (6,2) {$\hocofib(i|_{U_\alpha})$};
\draw (qUa) to node[above] {$i|_{U_\alpha}$} (pUa);
\draw (pUa) to (hcCa);

\node (S) at (0,0) {$\mathbb{A}^{n} \setminus \{0\} \times \tilde{U}_\alpha$};
\node (A) at (3,0) {$\mathbb{A}^{n} \times \tilde{U}_\alpha$};
\node (Q) at (6,0) {$\Sigma_s \left( \mathbb{A}^{n} \setminus \{0\}\right)$};
\draw (S) to (A);
\draw (A) to (Q);

\draw (S) to node[left] {$l^\ast i$} (qUa);
\draw (A) to node[left] {$p^\ast l$} (pUa);
\draw (Q) [dashed] to (hcCa);
\end{tikzpicture}\\
}
whose rows are homotopy cofiber sequences
and the leftmost two columns are weak equivalences.
Consequently, the last column is a weak equivalence.
This exhibits both $E \setminus B$ and $\Th(p)$ as homotopy colimit over cellular spaces.

We can view
$E \setminus B \to B$ obtained by composing $E \setminus B \to E$ with the bundle projection $p \colon E \to B$
as an explicit gluing map, as its homotopy colimit is again $\Th(p)$. 
\end{proof}

\begin{corollary}\label{cor:applied-hopurity-cool}
Given a sequence $M_i$ of smooth varieties over a ring $R$ which is smooth and finite type over a Dedekind ring with perfect residue fields or a field
\[M_n \supset M_{n-1} \supset \cdots \supset M_0 = \ast \supset M_{-1} = \emptyset\]
such that each $M_i$ is atacc~(\cref{defn:totally-affinely-contractible})
and rank $r_i$ vector bundles $V_i \epito M_{i-1}$ together with a closed immersion of the total space $V_i \closedto M_i$ of codimension $c_i$,
and each complement $X_i := M_i \setminus V_i$ is $\Ao$-contractible,
there exists an unstable motivic cell structure on each $M_i$.
\end{corollary}
\begin{proof}
We use induction on $i$, with the base case $M_0$ being trivially cellular.
Let $N_i \epito V_i$ be the normal bundle of the closed immersion $V_i \closedto M_i$.
As the complement $X_i$ is $\Ao$-contractible, by \cref{thm:hopurity}, applied as in the proof of \cref{cor:hopurity-first-application},
we get a weak equivalence $M_i \to \Th(N_i)$.
From our assumptions, $\Th(N_i)$ carries an unstable cell structure.
\end{proof}

\begin{remark}
If the ranks $r_i$ in \cref{cor:applied-hopurity-cool} are all $0$,
this resembles Wendt's unstable cell structure on generalized flag varieties using the Bruhat cells~\cite[Proposition 3.7]{wendtcell10}.
\end{remark}

We now prove that the complements of subspace arrangements are cellular after a finite amount of suspensions, depending on the number of subspaces.
\begin{proof}[{Proof of \cref{thm:subspace-arrangements}}]\label{proofof:subspace-arrangements}
  The base case is $\mathbb{A}^n$ with a single linear subspace $L_0$ of dimension $k$.
  By change of basis we move the linear subspace to the first $k$ coordinates so that
  $\mathbb{A}^n \setminus L_0 \isoto \left( \mathbb{A}^{k} \times \mathbb{A}^{n-k} \right) \setminus \mathbb{A}^k \times \{0\} \ =\ \mathbb{A}^k \times \left( \mathbb{A}^{n-k} \setminus \{0\}\right)$, an affine bundle over the motivic sphere $\mathbb{A}^{n-k} \setminus \{0\}$.

  By induction over the number of linear subspaces, assume $X := \mathbb{A}^n \setminus \bigcup_{i=1}^n L_i$ is $(n-1)$-suspended cellular. The intersection $L_0 \cap X = L_0 \setminus \bigcup_{i=1}^n (L_0 \cap L_i)$ is the complement of an arrangement of $n$ linear subspaces in the affine space $L_0$, hence $(n-1)$-suspended cellular by induction assumption.
  
  The normal bundle $N_0 \epito L_0 \cap X$ of $L_0 \cap X \closedto X$ is the restriction of the normal bundle of $L_0 \closedto \mathbb{A}^n$, hence trivial by \cref{thm:quillen-suslin}. Let $r_0 = n - \dim(L_o)$ be the rank of $N_0$. From this we see that the Thom space $\Th(N_0)$ is $\Ao$-homotopy equivalent to $\Sigma^{2k,k} \left(L_0 \cap X\right) = \Sigma^k\left( \left(\Gm\right)^{\wedge k} \wedge \left(L_0 \cap X\right) \right)$, hence $(n-1-k)$-suspended cellular.
  
  By \cref{thm:hopurity} we get a homotopy cofiber sequence
  \[ \mathbb{A}^n \setminus \bigcup_{i=0}^n H_i \to \mathbb{A}^n \setminus \bigcup_{i=1}^n H_i \to \Th(N_0) \to \Sigma \left( \mathbb{A}^n \setminus \bigcup_{i=0}^n H_i \right)\]
 which shows that the fourth space is $(n-1)$-suspended cellular as well. 
\end{proof}
A rank $r$ split torus is the complement of the $r$ coordinate hyperplanes $\{x_i = 0\} \subset \mathbb{A}^n$, hence we have shown that rank $r$ tori are $(r-1)$-suspended cellular.

Closely related to cellularity is the notion of \emph{linear varieties} which comes in several closely related versions (discussed by Janssen, Totaro, Joshua among others). We study two of them here.
\begin{defn}
  We call the empty scheme $\emptyset$ and any affine space $\mathbb{A}^n$ a $0$-linear variety.
  Inductively, for $n \in \Natural$ and $Z$ a $(n-1)$-linear variety with a closed immersion $Z \closedto X$ and open complement $U := X \setminus Z$, if either $X$ or $U$ is $(n-1)$-linear as well, then we call $Z,X,U$ \emph{$n$-linear} varieties.
  A variety is called \emph{linear} if it is $n$-linear for some $n$.
\end{defn}
\begin{lemma}\label{lemma:tori-linear}
A variety $Z$ that is isomorphic to a union of $r$ hyperplanes in $\mathbb{A}^n$ is $(r-1)$-linear and the complement in $\mathbb{A}^n$ is $r$-linear.
\end{lemma}
\begin{proof}
 Assume $Z \closedto \mathbb{A}^n$ to be a union of $r$ hyperplanes and $H \closedto \mathbb{A}^n$ another hyperplane.
 Then $Z \closedto Z \cup H$ has complement $H \setminus Z$ isomorphic to the union of $r-1$ hyperplanes. By induction over $r$ we can assume $Z$ and $H \setminus Z$ to be $(r-1)$-linear, hence $Z \cup H$ is $r$-linear.
Since $\mathbb{A}^n$ is $0$-linear, it is also $r-1$-linear, so $\mathbb{A}^n \setminus Z$ is $r$-linear.
\end{proof}

\begin{defn}
Let $X$ be a $k$-variety with a filtration $F^i$ such that $F^i \closedto F^{i+1}$ is a closed immersion with open complement isomorphic to a disjoint union of varieties of the type $\mathbb{A}^n \times \Gm^{\times r}$. Then $X$ is called \emph{very linear}.
\end{defn}
\begin{proposition}
Very linear varieties are linear.
\end{proposition}
\begin{proof}
Products and disjoint unions of linear varieties are again linear and the previous~\cref{lemma:tori-linear} shows that strata of the filtration are linear.
\end{proof}

\begin{remark}
While it is an easy exercise to show that the Voevodsky motive of a linear variety is of mixed Tate type,
it is not clear whether linear varieties are cellular,
as the Thom spaces involved in a homotopy purity argument are not necessarily cellular.
\end{remark}

\section{Cell Structures for Spherical Varieties}\label{section:spherical}

Let $G$ be a split reductive group. We will use the theory of spherical varieties, as detailed in the comprehensive book by Timashev~\cite{timashev11}.
A \emph{spherical variety} is a normal algebraic variety with an algebraic $G$-action such that a Borel $B \subset G$ acts with a dense orbit. As a consequence it has only finitely many $B$-orbits.

Special cases of spherical varieties are spherical homogeneous spaces such as affine quadrics, flag varieties $G/P$ and the homogeneous spaces $G/H$ that admit a two-orbit equivariant completion.

\begin{proposition}\label{prop:spherical-very-linear}
 $G$-spherical varieties are very linear.
\end{proposition}
\begin{proof}
  By Rosenlicht~\cite[Theorem 5, page 119]{rosenlicht63},
any $k$-variety homogeneous under $B$ is isomorphic to $\Gm^\times{r} \times \mathbb{A}^n$ for some $r,n$. The $B$-orbit decomposition of a spherical variety $X$ therefore yields a filtration $F^i$ with strata the $B$-orbits, turning $X$ into a very linear variety.
\end{proof}

We now prove that every spherical variety admits stable motivic cell structures (composing ideas from Totaro~\cite[page 8, section 3 and Addendum]{totaro14a} and Carlsson--Joshua~\cite[Proposition 4.7]{carlssonjoshua11}).
\begin{proof}[Proof of \cref{thm:spherical-stably-cellular}]\label{proofof:spherical-stably-cellular}
  The filtration $F^i$ by $B$-orbits, as in the proof of~\cref{prop:spherical-very-linear} has the special property that the $B$-orbits are atacc, hence for each $\iota_i \colon F^{i-1} \closedto F^i$ the homotopy cofiber sequence~\cref{thm:hopurity}
  \[ F^i \setminus F^{i-1} \to F^i \to \Th(N_{\iota_i})\]
  has $F^{i-1}$ and $\Th(N_{\iota_i})$ stably cellular by induction.
\end{proof}

\begin{remark}\label{remark:easy-spherical-cells}
In the special case of wonderful completions, there is another proof of~\cref{thm:spherical-stably-cellular}:
  Let $X$ be a homogeneous spherical $G$-variety which admits a \emph{wonderful} equivariant completion $\overline{X}$ with boundary $Z$, i.e. $\overline{X}$ is smooth,
\[G/H = X \opento \overline{X} \closedfrom Z,\]
the boundary $Z$ has $r$ irreducible components, where $r$ is the rank of $X$,
and there is a unique closed orbit in $Z$, which is the intersection of all irreducible components of $Z$.
Furthermore, all open orbits of $Z$ are of lower dimension than $X$~\cite[Chapter 5, Definition 30.1]{timashev11}.

As $\overline{X}$ is a complete $G$-variety, we can apply the algebraic Morse theory of Bia{\l}ynicky-Birula, 
as Wendt proved~\cite[Corollary 3.5]{wendtcell10}, to obtain a stable motivic cell structure on $\overline{X}$.
The same applies to $Z$, so by a $2$-out-of-$3$-argument, as in the previous proof,
the variety $X$ is stably motivically cellular.
\end{remark}

One can control the amount of suspensions one has to perform to obtain an unstable cell structure, provided the boundary is atacc:
\begin{proof}[{Proof of~\cref{thm:cells-from-completion}}]\label{proofof:cells-from-completion}
  The case of $n=1$ is the situation $X \opento Y \closedfrom D$ with $D$ irreducible, smooth and attac.
  By \cref{thm:unstable-thom-cells} the space $\Th(N_{D \monoto Y})$ is unstably cellular,
  hence $\Sigma^k \Th(N_{D\monoto Y})$ is unstably cellular. 
  By assumption, $\Sigma^k Y$ is unstably cellular. By taking the $k$-fold suspension on the \cref{thm:hopurity} cofiber sequence and adding the next term, we get
  \[\Sigma^k X \to \Sigma^k Y \to \Sigma^k \Th(N_{D\monoto Y}) \to \Sigma^{k+1} X\]
  which shows that $\Sigma^{k+1} X$ is unstably cellular.

  For arbitrary $n$ we form $X' := Y \setminus D_n$ and the inclusion $X' \opento Y$ is again the $n=1$ case we just proved,
  so that $\Sigma^{k+1} X'$ is unstably cellular.
  The inclusion $X \opento X'$ has boundary $D \setminus D_n$,
  so by induction over $n$ we get that $\Sigma^{k+n} X$ is unstably cellular.
\end{proof}

As a pleasant surprise, the rank $1$ situation is particularly well-behaved:
\begin{theorem}\label{thm:unstable-cells-after-single-suspension}
  For homogeneous spaces $X$ that admit an equivariant completion $\overline{X}$ with a single closed orbit as boundary (two-orbit wonderful completions),
  there exists an unstable motivic cell structure on $\Sigma X$.
\end{theorem}
\begin{proof}
  By the classification of such $X \opento \overline{X}$ (due to Ahiezer in characteristics $0$~\cite{ahiezer83} and Knop in positive characteristics~\cite{knop14}),
  we know that $\overline{X}$ is a projective homogeneous space under a reductive group.
  We apply Wendt's unstable motivic cell structure~\cite[Theorem 3.6]{wendtcell10} arising from the Bruhat decomposition of $\overline{X}$ and $\overline{X} \setminus X$ (both are generalized flag varieties).
  Since the Bruhat cells of a generalized flag variety are affine spaces, the boundary $\overline{X} \setminus X$ is atacc.
  Now~\cref{thm:cells-from-completion} gives the conclusion.
\end{proof}
\begin{remark}
  For each such two-orbit completions, there exists a choice of Borel for the reductive group acting transitively on the completion such that the associated Schubert stratification restricts to a Schubert stratification on the boundary. One can compute explicitly the embedding and the complement for each Bruhat cell, where the embedding restricts to a linear embedding of an affine subspace into an affine space. With this observation, one can also use the complements, which are $1$-suspended cellular due to \cref{thm:subspace-arrangements}, to give a slightly different proof of \cref{thm:unstable-cells-after-single-suspension}.
\end{remark}

\begin{example}\label{example:hpn}
For $\HP^n = \Sp_{2n+2}/\Sp_{2}\times\Sp_{2n}$, the completion is a Grassmannian $\Gr(2,2n+2)$,
with complement the symplectic Grassmannian $\SpGr(2,2n+2)$ classifying symplectic planes in a $2n+2$-dimensional vector space with the standard symplectic form.
\Cref{thm:unstable-cells-after-single-suspension} hands us an unstable motivic cell structure for $\Sigma \HP^n$:
\[\Gr(2,2n+2) \to \Th\left(N_{\SpGr(2,2n+2) \monoto \Gr(2,2n+2)}\right) \to \Sigma \HP^n.\]
\end{example}
As mentioned in the introduction, there also exists a motivic cell structure for $\HP^n$ \emph{unsuspended}~\cite[Theorem 4.4.8]{voelkel16}.

\begin{example}\label{example:sigma-op2}
For the split octonionic projective plane $\OP^2 := \Ffour/\Spin_9$, the completion is the complex Cayley plane $\Esix/P_1$,
with complement an $\Ffour/P_4$.
\Cref{thm:unstable-cells-after-single-suspension} hands us an unstable motivic cell structure for $\Sigma \OP^2$:
\[\Esix/P_1 \to \Th\left(N_{\Ffour/P_4}\right) \to \Sigma \OP^2.\]
Unfortunately, there is still no known construction of a motivic cell structure for $\OP^2$ unsuspended. Attempts on a construction are in the author's thesis~\cite[Section 4.6]{voelkel16}.
\end{example}
%
%
%
%
\hbadness=3200
\interlinepenalty=10000

%
\end{document}